\documentclass{article}%
\usepackage{amssymb}
\usepackage{amsmath}
\usepackage{amsfonts}
\usepackage{graphicx}%
\providecommand{\U}[1]{\protect\rule{.1in}{.1in}}
\newtheorem{theorem}{Theorem}[section]

\newtheorem{conjecture}[theorem]{Conjecture}
\newtheorem{corollary}[theorem]{Corollary}

\newtheorem{example}[theorem]{Example}

\newtheorem{lemma}[theorem]{Lemma}

\newenvironment{proof}[1][Proof]{\noindent\textbf{#1.} }{\ \rule{0.5em}{0.5em}}
\begin{document}

\author{Vadim E. Levit\\Department of Computer Science\\Ariel University, Israel\\levitv@ariel.ac.il
\and Eugen Mandrescu\\Department of Computer Science\\Holon Institute of Technology, Israel\\eugen\_m@hit.ac.il}
\title{On an Annihilation Number Conjecture}
\date{}
\maketitle

\begin{abstract}
Let $\alpha(G)$ denote the cardinality of a maximum independent set, while
$\mu(G)$ be the size of a maximum matching in the graph $G=\left(  V,E\right)
$. If $\alpha(G)+\mu(G)=\left\vert V\right\vert $, then $G$ is a
\textit{K\"{o}nig-Egerv\'{a}ry graph}. If $d_{1}\leq d_{2}\leq\cdots\leq
d_{n}$ is the degree sequence of $G$, then the \textit{annihilation number}
$h\left(  G\right)  $ of $G$ is the largest integer $k$ such that
$\sum\limits_{i=1}^{k}d_{i}\leq\left\vert E\right\vert $
\cite{Pepper2004,Pepper2009}. A set $A\subseteq V$ satisfying $\sum
\limits_{a\in A}\deg(a)\leq\left\vert E\right\vert $ is an
\textit{annihilation set}; if, in addition, $\deg\left(  v\right)
+\sum\limits_{a\in A}\deg(a)>\left\vert E\right\vert $, for every vertex $v\in
V(G)-A$, then $A$ is a \textit{maximal annihilation set} in $G$.

In \cite{LarsonPepper2011} it was conjectured that the following assertions
are equivalent:

\textit{(i) }$\alpha\left(  G\right)  =h\left(  G\right)  $;

\textit{(ii)} $G$\ is a K\"{o}nig-Egerv\'{a}ry graph and every maximum
independent set is a maximal annihilating set.

In this paper, we prove that the implication \textquotedblright\textit{(i)}
$\Longrightarrow$ \textit{(ii)}\textquotedblleft\ is correct, while for the
opposite direction we provide a series of generic counterexamples.


\end{abstract}

\section{Introduction}

Throughout this paper $G=(V,E)$ is a finite, undirected, loopless graph
without multiple edges, with vertex set $V=V(G)$ of cardinality $\left\vert
V\left(  G\right)  \right\vert =n\left(  G\right)  $, and edge set $E=E(G)$ of
size $\left\vert E\left(  G\right)  \right\vert =m\left(  G\right)  $. If
$X\subset V(G)$, then $G[X]$ is the subgraph of $G$ induced by $X$. By $G-v$
we mean the subgraph $G[V(G)-\left\{  v\right\}  ]$, for $v\in V(G)$.
$K_{n},K_{m,n},P_{n},C_{n}$ denote respectively, the complete graph on
$n\geq1$ vertices, the complete bipartite graph on $m,n\geq1$ vertices, the
path on $n\geq1$ vertices, and the cycle on $n\geq3$ vertices, respectively.

The \textit{disjoint union} of the graphs $G_{1},G_{2}$ is the graph
$G_{1}\cup G_{2}$ having the disjoint union of $V(G_{1}),V(G_{2})$ as a vertex
set, and the disjoint union of $E(G_{1}),E(G_{2})$ as an edge set. In
particular, $nG$ denotes the disjoint union of $n>1$ copies of the graph $G$.

A set $S\subseteq V(G)$ is \textit{independent} if no two vertices from $S$
are adjacent, and by $\mathrm{Ind}(G)$ we mean the family of all the
independent sets of $G$. An independent set of maximum size is a
\textit{maximum independent set} of $G$, and $\alpha(G)=\max\{\left\vert
S\right\vert :S\in\mathrm{Ind}(G)\}$. Let $\Omega(G)$ denote the family of all
maximum independent sets.

A \textit{matching} in a graph $G$ is a set of edges $M\subseteq E(G)$ such
that no two edges of $M$ share a common vertex. A matching of maximum
cardinality $\mu(G)$ is a \textit{maximum matching}, and a \textit{perfect
matching} is one saturating all vertices of $G$.

It is known that $\lfloor\left\vert V\left(  G\right)  \right\vert
/2\rfloor+1\leq\alpha(G)+\mu(G)\leq\left\vert V\left(  G\right)  \right\vert
\leq\alpha(G)+2\mu(G)$ hold for every graph $G$ \cite{BGL2002}. If
$\alpha(G)+\mu(G)=\left\vert V\left(  G\right)  \right\vert $, then $G$ is
called a K\"{o}nig-Egerv\'{a}ry graph\textit{ }\cite{dem,ster}. For instance,
each bipartite graph is a K\"{o}nig-Egerv\'{a}ry graph \cite{eger,koen}.
Various properties of K\"{o}nig-Egerv\'{a}ry graphs can be found in
\cite{BeckBorn2018,Bhattacharya2018,Bonomo2013,JarLevMan2017,JarLevMan2018,JarLevMan2019,Korach2006,Larson2007,Larson2011,LevMan2002,LevMan2003,LevMan2012a,LevMan2012b,LevMan2013,LevManLemma2014,LevMan2019,Short2015}%
.

Let $d_{1}\leq d_{2}\leq\cdots\leq d_{n}$ be the degree sequence of a graph
$G$. Pepper \cite{Pepper2004,Pepper2009} defined the annihilation number of
$G$, denoted $h\left(  G\right)  $, to be the largest integer $k$ such that
the sum of the first $k$ terms of the degree sequence is at most half the sum
of the degrees in the sequence. In other words, $h\left(  G\right)  $ is
precisely the largest integer $k$ such that $\sum\limits_{i=1}^{k}d_{i}\leq
m\left(  G\right)  $.

Clearly, $h\left(  G\right)  =n\left(  G\right)  $ if and only if $m\left(
G\right)  =0$, while $h\left(  G\right)  =n\left(  G\right)  -1$ if and only
if $m\left(  G\right)  =1$. Moreover, for every positive integer $p$, there
exists a connected graph, namely $K_{1,p}$, having $h\left(  K_{1,p}\right)
=$ $p=n\left(  K_{1,p}\right)  -1$.

For $A\subseteq V(G)$, let $\deg(A)=\sum\limits_{v\in A}\deg(v)$. Every
$A\subseteq V\left(  G\right)  $ satisfying $\deg(A)\leq m\left(  G\right)  $
is an \textit{annihilating set}. Clearly, every independent set is
annihilating. An annihilating set $A$ is \textit{maximal} if $\deg
(A\cup\{v\})>m\left(  G\right)  $, for every vertex $v\in V(G)-A$, and it is
\textit{maximum} if $\left\vert A\right\vert =h\left(  G\right)  $
\cite{Pepper2004}. For example, if $G=K_{p,q}=(A,B,E)$ and $p>q$, then $A$ is
a maximum annihilating set, while $B$ is a maximal annihilating set.

Recall that a \textit{vertex-cover} of a graph $G$ is a subset of vertices
$W\subseteq V\left(  G\right)  $ such that $W\cap\left\{  u,v\right\}
\neq\emptyset$ holds for every $uv\in E\left(  G\right)  $. Notice that a
vertex cover need not to be independent.

\begin{lemma}
\label{lem7}Every independent vertex cover of a graph without isolated
vertices is a maximal annihilating set.
\end{lemma}

\begin{proof}
Let $A$ be an independent vertex cover for a graph $G$. Since $A$ is
independent, we get that $\deg\left(  A\right)  =m\left(  G\right)  $. This
ensures that $A$ is a maximal annihilating set, because $G$ has no isolated vertices.
\end{proof}

\begin{theorem}
\cite{Pepper2004}\label{th4} For every graph $G$, $h\left(  G\right)  \geq
\max\left\{  \left\lfloor \frac{n\left(  G\right)  }{2}\right\rfloor
,\alpha\left(  G\right)  \right\}  $.
\end{theorem}

For instance, $h\left(  C_{7}\right)  =\alpha\left(  C_{7}\right)
=\left\lfloor \frac{n\left(  C_{7}\right)  }{2}\right\rfloor $, $h\left(
\overline{P_{5}}\right)  =3>\alpha\left(  \overline{P_{5}}\right)
=\left\lfloor \frac{n\left(  \overline{P_{5}}\right)  }{2}\right\rfloor $,

$h\left(  K_{2,3}\right)  =\alpha\left(  K_{2,3}\right)  >\left\lfloor
\frac{n\left(  K_{2,3}\right)  }{2}\right\rfloor $, while $h\left(
\overline{C_{6}}\right)  =\left\lfloor \frac{n\left(  \overline{C_{6}}\right)
}{2}\right\rfloor >\alpha\left(  \overline{C_{6}}\right)  $.

The relation between the annihilation number and various parameters of a graph
were studied in
\cite{amjadi2015,Arqm2018,Dehgardi2013a,Dehgardi2013b,Dehgardi2014,Desormeaux2013,Gentner2017,Jaumea2018,Pepper2004}%
.

\begin{theorem}
\label{Th7}\cite{LarsonPepper2011} For a graph $G$ with $h\left(  G\right)
\geq\frac{n\left(  G\right)  }{2}$, $\alpha\left(  G\right)  =h\left(
G\right)  $ if and only if $G$\ is a K\"{o}nig-Egerv\'{a}ry graph and every
$S\in\Omega(G)$ is a maximum annihilating set.
\end{theorem}

Actually, Larson and Pepper \cite{LarsonPepper2011} proved a stronger result
that reads as follows.

\begin{theorem}
Let $G$ be a graph with $h\left(  G\right)  \geq\frac{n\left(  G\right)  }{2}%
$. Then the following are equivalent:

\emph{(i)} $\alpha\left(  G\right)  =h\left(  G\right)  $;

\emph{(ii)} $G$\ is a K\"{o}nig-Egerv\'{a}ry graph and every $S\in\Omega(G)$
is a maximum annihilating set;

\emph{(iii)} $G$\ is a K\"{o}nig-Egerv\'{a}ry graph and some $S\in\Omega(G)$
is a maximum annihilating set.
\end{theorem}

Along these lines, it was conjectured that the impacts of maximum and maximal
annihilating sets are the same.

\begin{conjecture}
\label{Conj1}\cite{LarsonPepper2011} Let $G$ be a graph with $h\left(
G\right)  \geq\frac{n\left(  G\right)  }{2}$. Then the following assertions
are equivalent:

\emph{(i)} $\alpha\left(  G\right)  =h\left(  G\right)  $;

\emph{(ii)} $G$\ is a K\"{o}nig-Egerv\'{a}ry graph and every $S\in\Omega(G)$
is a maximal annihilating set.
\end{conjecture}

In this paper we validate the \textquotedblright\textit{(i)} $\Longrightarrow$
\textit{(ii)}\textquotedblleft\ part of Conjecture \ref{Conj1}, while for the
converse, we provide some generic counterexamples, including trees, bipartite
graphs that are not trees, and non-bipartite K\"{o}nig-Egerv\'{a}ry graphs.
Let us notice that, if $G$ is a K\"{o}nig-Egerv\'{a}ry graph, bipartite or
non-bipartite, and $H=qK_{1}\cup G$, then $H$ inherits these properties.
Moreover, the relationship between the independence numbers and annihilation
numbers of $G$ and $H$ remains the same, because $\alpha\left(  H\right)
=\alpha\left(  G\right)  +q$ and $h\left(  H\right)  =h\left(  G\right)  +q$.
Therefore, it is enough to construct only connected counterexamples.

Finally, we hypothesize that Conjecture \ref{Conj1} is true for connected
graphs with independence number equal to three.

\section{The annihilation number of a sequence}

Let $D=\left(  d_{1},d_{2},\cdots,d_{n}\right)  $ be a sequence of real
numbers such that $d_{1}\leq d_{2}\leq\cdots\leq d_{n}$. The
\textit{annihilation number} $h\left(  D,\Theta\right)  $ of $D$ with respect
to the threshold $\Theta$ is the largest integer $k$ such that $\sum
\limits_{i=1}^{k}d_{i}\leq\Theta$. A subsequence $A$ of $D$ is
\textit{maximum} if $\left\vert A\right\vert =h\left(  D,\Theta\right)  $.
\ For a subsequence $A$ of $D$, let $\deg(A)=\sum\limits_{i\in Dom\left(
A\right)  }d_{i}$, where $Dom\left(  A\right)  $ is the set of all indexes
defining $A$. If $\deg(A)\leq\Theta$, then $A$ is an \textit{annihilating
subsequence} of $D$. An annihilating subsequence $A$ is \textit{maximal} if
$\deg(A)+d_{i}>\Theta$ for every index $i\notin Dom\left(  A\right)  $.

\begin{example}
Let $D=\left(  1,2,3,4,4\right)  $. For $\Theta=3$, we get $h\left(
D,\Theta\right)  =2$ and $A=(1,2)$ is a maximum subsequence of $D$. For
$\Theta=6$, we get $h\left(  D,\Theta\right)  =3$ and $A_{1}=(1,2,3)$ is a
maximum subsequence of $D$, while $A_{2}=(2,4)$ is a maximal non-maximum
subsequence of $D$.
\end{example}

\begin{theorem}
\label{Th8}Let $D=\left(  d_{1},d_{2},\cdots,d_{n}\right)  $ be a sequence of
real numbers such that
\[
d_{1}\leq d_{2}\leq\cdots\leq d_{n}\text{, and }h\left(  D,\Theta\right)  =k.
\]
Then every maximum annihilating subsequence of $D$ is maximal as well.
\end{theorem}

\begin{proof}
Clearly, $A=\left(  d_{1},d_{2},\cdots,d_{k}\right)  $ is both a maximum and a
maximal subsequence of $D$, as $h\left(  D,\Theta\right)  =k$.

Let
\[
B=\left(  d_{j_{1}},d_{j_{2}},\cdots,d_{j_{k}}\right)  ,j_{1}<j_{2}%
<\cdots<j_{k}%
\]
be a maximum annihilating subsequence of $D$. Suppose, to the contrary, that
$B$ is not maximal, i.e., there exists some index $q$ such that $d_{j_{1}%
}+d_{j_{2}}+\cdots+d_{j_{k}}+d_{q}\leq\Theta$.

If $q>k$, then
\[
d_{j_{1}}+d_{j_{2}}+\cdots+d_{j_{k}}+d_{q}\geq d_{1}+d_{2}+\cdots+d_{k}%
+d_{q}>\Theta.
\]
Therefore, $q\leq k$. In what follows, without loss of generality, we may
assume that $q$ is the smallest index possible satisfying this inequality
$d_{j_{1}}+d_{j_{2}}+\cdots+d_{j_{k}}+d_{q}\leq\Theta$. Hence, we infer that
\[
d_{1}=d_{j_{1}},d_{2}=d_{j_{2}},\cdots,d_{q-1}=d_{j_{q-1}}\text{, and }q+1\leq
j_{q}.
\]

Let $C=\left(  d_{1},d_{2},\cdots d_{q-1},d_{q},d_{j_{q}},d_{j_{q+1}}%
,\cdots,d_{j_{k-1}}\right)  $. Then we have
\begin{gather*}
d_{j_{1}}+d_{j_{2}}+\cdots+d_{j_{k}}+d_{q}=\\
d_{1}+d_{2}+\cdots+d_{q-1}+\mathbf{d}_{q}+d_{j_{q}}+d_{j_{q+1}}+\cdots
+d_{j_{k-1}}+d_{j_{k}}\geq\\
d_{1}+d_{2}+\cdots+d_{q-1}+\mathbf{d}_{q}+d_{q+1}+d_{q+2}+\cdots
+d_{k}+d_{j_{k}}>\Theta,
\end{gather*}
which contradicts the assumption that $B$ is not maximal.
\end{proof}

Let us suppose that $D=\left(  d_{1},d_{2},\cdots,d_{n}\right)  $ is a degree
sequence of a graph $G$, i.e., $d_{i}=\deg(v_{i}),v_{i}\in V\left(  G\right)
$, and $\Theta=m=\left\vert E\left(  G\right)  \right\vert $. In this case,
the definition of the annihilation number of a sequence coincides with the
original Pepper's definition of the annihilation number of the graph $G$. In
what follows, we apply Theorem \ref{Th8} to graphs.

The cycle $C_{5}$ has $h\left(  C_{5}\right)  =\alpha\left(  C_{5}\right)
=\left\lfloor \frac{n\left(  C_{5}\right)  }{2}\right\rfloor $ and every of
its maximum independent sets is both a maximal and a maximum annihilating set.
Notice that $C_{5}$\ is not a K\"{o}nig-Egerv\'{a}ry graph. \begin{figure}[h]
\setlength{\unitlength}{1cm}\begin{picture}(5,1)\thicklines
\multiput(3.5,1)(2,0){2}{\circle*{0.29}}
\multiput(3.5,0)(1,0){3}{\circle*{0.29}}
\put(3.5,1){\line(1,-1){1}}
\put(3.5,0){\line(1,0){2}}
\put(3.5,0){\line(0,1){1}}
\put(4.5,0){\line(1,1){1}}
\put(5.5,0){\line(0,1){1}}
\put(2.7,0.5){\makebox(0,0){$G_{1}$}}
\multiput(8,0)(1,0){3}{\circle*{0.29}}
\multiput(8,1)(2,0){2}{\circle*{0.29}}
\put(8,0){\line(1,0){2}}
\put(8,0){\line(0,1){1}}
\put(8,0){\line(2,1){2}}
\put(8,1){\line(1,-1){1}}
\put(8,1){\line(2,-1){2}}
\put(9,0){\line(1,1){1}}
\put(10,0){\line(0,1){1}}
\put(7.2,0.5){\makebox(0,0){$G_{2}$}}
\end{picture}\caption{Non-K\"{o}nig-Egerv\'{a}ry graphs with $h\left(
G_{1}\right)  =3$ and $h\left(  G_{2}\right)  =2$.}%
\label{fig3}%
\end{figure}
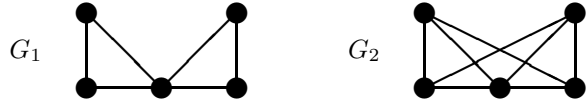

Consider the graphs from Figure \ref{fig3}. The graph $G_{1}$ has $h\left(
G_{1}\right)  >\alpha\left(  G_{1}\right)  $ and none of its maximum
independent sets is a maximal or a maximum annihilating set. The graph $G_{2}$
has $h\left(  G_{2}\right)  =\alpha\left(  G_{2}\right)  $ and each of its
maximum independent sets is both a maximal and a maximum annihilating set.
Notice that $h\left(  G_{1}\right)  >\frac{n\left(  G_{1}\right)  }{2}$, while
$h\left(  G_{2}\right)  <\frac{n\left(  G_{2}\right)  }{2}$.

Consider now the graphs from Figure \ref{fig333}. \ The graph $G_{1}$ has
$\alpha\left(  G_{1}\right)  =\frac{n\left(  G_{1}\right)  }{2}<h\left(
G_{1}\right)  $ and each of its maximum independent sets is neither a maximal
nor a maximum annihilating set. The graph $G_{2}$ has $h\left(  G_{2}\right)
=\alpha\left(  G_{2}\right)  =\frac{n\left(  G_{2}\right)  }{2}$, every of its
maximum independent sets is both a maximal and a maximum annihilating set, and
it has a maximal independent set that is a maximal non-maximum annihilating
set, namely $\left\{  a,b\right\}  $. The graph $G_{3}$ has $h\left(
G_{3}\right)  =\alpha\left(  G_{3}\right)  >\frac{n\left(  G_{3}\right)  }{2}$
and every of its maximum independent sets is both a maximal and maximum
annihilating set. The graph $G_{4}$ has $h\left(  G_{4}\right)  >\alpha\left(
G_{4}\right)  >\frac{n\left(  G_{4}\right)  }{2}$ and none of its maximum
independent sets is a maximal or maximum annihilating set.

\begin{figure}[h]
\setlength{\unitlength}{1cm}\begin{picture}(5,1.7)\thicklines
\multiput(0,1)(1,0){3}{\circle*{0.29}}
\multiput(0,0)(1,0){3}{\circle*{0.29}}
\put(0,0){\line(1,1){1}}
\put(0,0){\line(1,0){2}}
\put(0,0){\line(0,1){1}}
\put(1,0){\line(0,1){1}}
\put(1,0){\line(1,1){1}}
\put(2,0){\line(0,1){1}}
\put(1,1.5){\makebox(0,0){$G_{1}$}}
\multiput(3,0)(1,0){3}{\circle*{0.29}}
\multiput(3,1)(1,0){3}{\circle*{0.29}}
\put(3,0){\line(1,0){2}}
\put(3,0){\line(0,1){1}}
\put(4,0){\line(0,1){1}}
\put(4,0){\line(1,1){1}}
\put(5,0){\line(0,1){1}}
\qbezier(3,1)(4,1.5)(5,1)
\put(3.3,0.9){\makebox(0,0){$a$}}
\put(3.7,0.35){\makebox(0,0){$b$}}
\put(4,1.5){\makebox(0,0){$G_{2}$}}
\multiput(6,0)(1,0){4}{\circle*{0.29}}
\multiput(6,1)(1,0){3}{\circle*{0.29}}
\put(6,0){\line(1,0){3}}
\put(6,1){\line(1,-1){1}}
\put(7,0){\line(0,1){1}}
\put(7,1){\line(1,0){1}}
\put(8,1){\line(1,-1){1}}
\put(7,1.5){\makebox(0,0){$G_{3}$}}
\multiput(10,0)(1,0){4}{\circle*{0.29}}
\multiput(10,1)(1,0){4}{\circle*{0.29}}
\put(10,0){\line(1,0){3}}
\put(10,1){\line(1,-1){1}}
\put(11,0){\line(0,1){1}}
\put(11,0){\line(1,1){1}}
\put(11,1){\line(1,0){1}}
\put(11,1){\line(1,-1){1}}
\put(12,0){\line(0,1){1}}
\put(12,0){\line(1,1){1}}
\put(11.5,1.5){\makebox(0,0){$G_{4}$}}
\end{picture}\caption{K\"{o}nig-Egerv\'{a}ry graphs with $h\left(
G_{1}\right)  =h\left(  G_{3}\right)  =4$, $h\left(  G_{2}\right)  =3$,
$h\left(  G_{4}\right)  =6$.}%
\label{fig333}%
\end{figure}
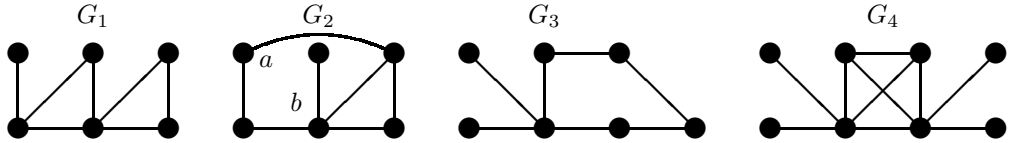

\begin{theorem}
\label{Th9}Let $G$ be a graph with $h\left(  G\right)  \geq\frac{n\left(
G\right)  }{2}$. If $\alpha\left(  G\right)  =h\left(  G\right)  $, then
$G$\ is a K\"{o}nig-Egerv\'{a}ry graph and every maximum independent set is a
maximal annihilating set.
\end{theorem}

\begin{proof}
Since $h\left(  G\right)  \geq\frac{n\left(  G\right)  }{2}$ and
$\alpha\left(  G\right)  =h\left(  G\right)  $, Theorem \ref{Th7} ensures that
$G$\ is a K\"{o}nig-Egerv\'{a}ry graph and each maximum independent set is a
maximum annihilating set. Hence, in accordance with Theorem \ref{Th8}, we
infer that every maximum independent set is a maximal annihilating set as well.
\end{proof}

Thus Conjecture \ref{Conj1} is half-verified.

\section{Tree counterexamples}

Recall that if a graph $G$ has a unique maximum independent set, say $S$, and
$V\left(  G\right)  -S$\ is also an independent set, then $G$ is a
\textit{strong unique maximum independence graph} \cite{Hop1984}.

\begin{theorem}
\cite{Hop1984} \label{th2}\emph{(i)} A tree is a \textit{strong unique maximum
independence tree (graph) if and only if the distance between any two leaves
is even.}

\emph{(ii)} A connected graph is a strong unique maximum independence graph if
and only if it is bipartite and has a spanning tree which is a \textit{strong
unique maximum independence tree (graph). }
\end{theorem}

\begin{figure}[h]
\setlength{\unitlength}{1cm}\begin{picture}(5,1.5)\thicklines
\multiput(3,0.5)(1,0){8}{\circle*{0.29}}
\multiput(3,1.5)(1,0){9}{\circle*{0.29}}
\put(3,0.5){\line(0,1){1}}
\put(3,0.5){\line(1,1){1}}
\put(4,0.5){\line(0,1){1}}
\multiput(4,0.5)(0.2,0){5}{\circle*{0.07}}
\put(5,0.5){\line(0,1){1}}
\put(5,0.5){\line(1,1){1}}
\put(6,0.5){\line(0,1){1}}
\put(6,0.5){\line(1,1){1}}
\put(7,0.5){\line(0,1){1}}
\put(7,0.5){\line(4,1){4}}
\put(7,1.5){\line(1,-1){1}}
\put(7,1.5){\line(2,-1){2}}
\put(7,1.5){\line(3,-1){3}}
\put(8,0.5){\line(1,1){1}}
\put(8,1.5){\line(1,-1){1}}
\put(10,0.5){\line(0,1){1}}
\put(3,1.85){\makebox(0,0){$x_{k}$}}
\put(4,1.85){\makebox(0,0){$x_{k-1}$}}
\put(5,1.85){\makebox(0,0){$x_{2}$}}
\put(6,1.85){\makebox(0,0){$x_{1}$}}
\put(3,0.15){\makebox(0,0){$y_{k}$}}
\put(4,0.15){\makebox(0,0){$y_{k-1}$}}
\put(5,0.15){\makebox(0,0){$y_{2}$}}
\put(6,0.15){\makebox(0,0){$y_{1}$}}
\put(7,1.85){\makebox(0,0){$a_{5}$}}
\put(8,1.85){\makebox(0,0){$a_{4}$}}
\put(9,1.85){\makebox(0,0){$a_{3}$}}
\put(10,1.85){\makebox(0,0){$a_{2}$}}
\put(11,1.85){\makebox(0,0){$a_{1}$}}
\put(7,0.15){\makebox(0,0){$b_{4}$}}
\put(8,0.15){\makebox(0,0){$b_{3}$}}
\put(9,0.15){\makebox(0,0){$b_{2}$}}
\put(10,0.15){\makebox(0,0){$b_{1}$}}
\put(2,1){\makebox(0,0){$T_{k}$}}
\end{picture}\caption{Strong unique maximum independence trees $T_{k},k\geq
1$.}%
\label{fig122}%
\end{figure}
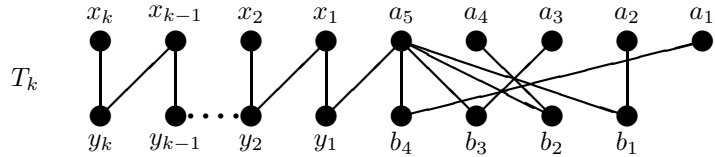Consider the trees from Figures \ref{fig55} and \ref{fig15}.

\begin{itemize}
\item $T_{1}$ has $m\left(  T_{1}\right)  =7$, the degree sequence $\left(
1,1,1,1,2,2,3,3\right)  $, $h\left(  T_{1}\right)  =5=\alpha\left(
T_{1}\right)  $, $\Omega\left(  T_{1}\right)  =\left\{  S_{1},S_{2}\right\}  $
, where $S_{1}=\left\{  v_{1},v_{2},v_{4},v_{7},v_{8}\right\}  $,
$S_{1}=\left\{  v_{1},v_{2},v_{5},v_{7},v_{8}\right\}  $, and $\deg\left(
S_{1}\right)  =\deg\left(  S_{2}\right)  =6$. Hence, every maximum independent
set of $T_{1}$ is both a maximal and a maximum annihilating set.

\item $T_{2}$ has $m\left(  T_{2}\right)  =5$, the degree sequence $\left(
1,1,1,2,2,3\right)  $, $h\left(  T_{2}\right)  =4>\alpha\left(  T_{2}\right)
$, $\Omega\left(  T_{2}\right)  =\left\{  S_{1},S_{2},S_{3},S_{4}%
,S_{5}\right\}  $, where $S_{1}=\left\{  a_{1},a_{2},a_{3}\right\}  $ has
$\deg\left(  S_{1}\right)  =5$, $S_{2}=\left\{  b_{1},b_{2},b_{3}\right\}  $
has $\deg\left(  S_{2}\right)  =5$, $S_{3}=\left\{  a_{1},a_{2},b_{3}\right\}
$ has $\deg\left(  S_{3}\right)  =4$, $S_{4}=\left\{  a_{1},b_{2}%
,a_{3}\right\}  $ has $\deg\left(  S_{4}\right)  =4$, $S_{5}=\left\{
a_{1},b_{2},b_{3}\right\}  $ has $\deg\left(  S_{5}\right)  =3$. Consequently,
only $S_{1}$ and $S_{2}$ are maximal annihilating sets belonging to
$\Omega\left(  T_{2}\right)  $.
\end{itemize}

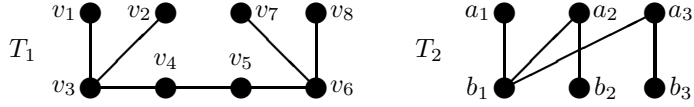
\begin{figure}[h]
\setlength{\unitlength}{1cm}\begin{picture}(5,1.2)\thicklines
\multiput(3,0)(1,0){4}{\circle*{0.29}}
\multiput(3,1)(1,0){4}{\circle*{0.29}}
\put(3,0){\line(0,1){1}}
\put(3,0){\line(1,1){1}}
\put(3,0){\line(1,0){3}}
\put(5,1){\line(1,-1){1}}
\put(6,0){\line(0,1){1}}
\put(2.65,1){\makebox(0,0){$v_{1}$}}
\put(3.65,1){\makebox(0,0){$v_{2}$}}
\put(2.65,0){\makebox(0,0){$v_{3}$}}
\put(4,0.35){\makebox(0,0){$v_{4}$}}
\put(5,0.35){\makebox(0,0){$v_{5}$}}
\put(6.35,0){\makebox(0,0){$v_{6}$}}
\put(5.35,1){\makebox(0,0){$v_{7}$}}
\put(6.35,1){\makebox(0,0){$v_{8}$}}
\put(2.1,0.5){\makebox(0,0){$T_{1}$}}
\multiput(8.5,0)(1,0){3}{\circle*{0.29}}
\multiput(8.5,1)(1,0){3}{\circle*{0.29}}
\put(8.5,0){\line(0,1){1}}
\put(8.5,0){\line(1,1){1}}
\put(8.5,0){\line(2,1){2}}
\put(9.5,0){\line(0,1){1}}
\put(10.5,0){\line(0,1){1}}
\put(8.15,1){\makebox(0,0){$a_{1}$}}
\put(8.15,0){\makebox(0,0){$b_{1}$}}
\put(9.85,1){\makebox(0,0){$a_{2}$}}
\put(9.85,0){\makebox(0,0){$b_{2}$}}
\put(10.85,1){\makebox(0,0){$a_{3}$}}
\put(10.85,0){\makebox(0,0){$b_{3}$}}
\put(7.5,0.5){\makebox(0,0){$T_{2}$}}
\end{picture}\caption{Trees with $\alpha\left(  T_{1}\right)  =5$ and
$\alpha\left(  T_{2}\right)  =3$.}%
\label{fig55}%
\end{figure}

\begin{itemize}
\item $T_{3}$ has $m\left(  T_{3}\right)  =10$, the degree sequence $\left(
1,1,1,1,1,1,2,2,3,3,4\right)  $, $h\left(  T_{3}\right)  =8>\alpha\left(
T_{3}\right)  $, $\Omega\left(  T_{3}\right)  =\left\{  S\right\}  $, where
$S=\left\{  x_{1},x_{2},x_{3},x_{4},x_{5},x_{6},x_{7}\right\}  $ has
$\deg\left(  S\right)  =8$. Thus $S$ is a (unique) maximum independent set but
not a maximal annihilating set.
\end{itemize}

\begin{figure}[h]
\setlength{\unitlength}{1cm}\begin{picture}(5,1.2)\thicklines
\multiput(5,0)(1,0){6}{\circle*{0.29}}
\multiput(5,1)(1,0){5}{\circle*{0.29}}
\put(5,0){\line(1,0){3}}
\put(9,0){\line(1,0){1}}
\put(5,1){\line(1,-1){1}}
\put(6,1){\line(1,-1){1}}
\put(7,0){\line(0,1){1}}
\put(8,0){\line(0,1){1}}
\put(9,0){\line(0,1){1}}
\put(8,0){\line(1,1){1}}
\put(4.62,0){\makebox(0,0){$x_{1}$}}
\put(4.62,1){\makebox(0,0){$x_{2}$}}
\put(5.62,1){\makebox(0,0){$x_{3}$}}
\put(6.62,1){\makebox(0,0){$x_{4}$}}
\put(8.37,1){\makebox(0,0){$x_{5}$}}
\put(9.37,1){\makebox(0,0){$x_{6}$}}
\put(10.37,0){\makebox(0,0){$x_{7}$}}
\put(3.8,0.5){\makebox(0,0){$T_{3}$}}
\end{picture}\caption{A tree with $\alpha\left(  T_{3}\right)  =7$.}%
\label{fig15}%
\end{figure}
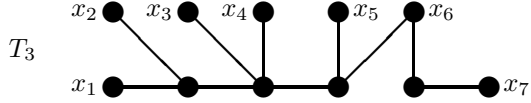

\begin{theorem}
\label{Th12}There exist a tree of order $2k+1,k\geq4$ and a tree of order
$2k+4,k\geq3$, satisfying the following conditions:

\begin{itemize}
\item $h\left(  T\right)  \geq\frac{n\left(  T\right)  }{2}$,

\item each $S\in\Omega\left(  T\right)  $ is a maximal non-maximum
annihilating set.
\end{itemize}
\end{theorem}

\begin{proof}
Let us consider the trees from Figure \ref{fig71}.

The tree $T_{2k+1},k\geq4$, has $n\left(  T_{2k+1}\right)  =2k+1$, $m\left(
T_{2k+1}\right)  =2k$, $\alpha\left(  T_{2k+1}\right)  =k+1$ and the degree
sequence%
\[
\underset{k}{\underbrace{1,\cdots,1}},\underset{k}{\underbrace{2,\cdots,2}%
},k.
\]
Consequently, we infer that $h\left(  T_{2k+1}\right)  =k+\left\lfloor
\frac{k}{2}\right\rfloor >\max\left\{  \frac{n\left(  T_{2k+1}\right)  }%
{2},\alpha\left(  T_{2k+1}\right)  \right\}  $. According to Theorem
\ref{th2}\emph{(i)}, $T_{2k+1}$ is a strong unique maximum independence
tree\textit{. }Hence, $\Omega\left(  T_{2k+1}\right)  =\left\{  S\right\}  $,
where $S=\left\{  a_{1},...,a_{k},v\right\}  $. Since $S$ is an independent
vertex cover, by Lemma \ref{lem7}, we conclude that each maximum independent
set of $T_{2k+1}$ is a maximal non-maximum annihilating set. \begin{figure}[h]
\setlength{\unitlength}{1cm}\begin{picture}(5,2,5)\thicklines
\put(3,1){\circle*{0.29}}
\multiput(4,0)(0,1){3}{\circle*{0.29}}
\multiput(5,0)(0,1){3}{\circle*{0.29}}
\multiput(4,1)(0,0.15){7}{\circle*{0.08}}
\multiput(5,1)(0,0.15){7}{\circle*{0.08}}
\put(3,1){\line(1,0){2}}
\put(3,1){\line(1,1){1}}
\put(3,1){\line(1,-1){1}}
\put(4,0){\line(1,0){1}}
\put(4,1){\line(1,0){1}}
\put(4,2){\line(1,0){1}}
\put(4.35,0.3){\makebox(0,0){$b_{1}$}}
\put(4.35,1.3){\makebox(0,0){$b_{2}$}}
\put(4.35,2.3){\makebox(0,0){$b_{k}$}}
\put(5.38,0){\makebox(0,0){$a_{1}$}}
\put(5.38,1){\makebox(0,0){$a_{2}$}}
\put(5.38,2){\makebox(0,0){$a_{k}$}}
\put(2.8,1.3){\makebox(0,0){$v$}}
\put(2,1){\makebox(0,0){$T_{2k+1}$}}
\multiput(7,0)(0,2){2}{\circle*{0.29}}
\multiput(8,1)(1,0){2}{\circle*{0.29}}
\multiput(10,0)(0,1){3}{\circle*{0.29}}
\multiput(11,0)(0,1){3}{\circle*{0.29}}
\multiput(10,1)(0,0.15){7}{\circle*{0.08}}
\multiput(11,1)(0,0.15){7}{\circle*{0.08}}
\put(7,2){\line(1,-1){1}}
\put(7,0){\line(1,1){1}}
\put(8,1){\line(1,0){1}}
\put(9,1){\line(1,1){1}}
\put(9,1){\line(1,0){1}}
\put(9,1){\line(1,-1){1}}
\put(10,0){\line(1,0){1}}
\put(10,1){\line(1,0){1}}
\put(10,2){\line(1,0){1}}
\put(10.35,0.3){\makebox(0,0){$b_{1}$}}
\put(10.35,1.3){\makebox(0,0){$b_{2}$}}
\put(10.35,2.3){\makebox(0,0){$b_{k}$}}
\put(11.38,0){\makebox(0,0){$a_{1}$}}
\put(11.38,1){\makebox(0,0){$a_{2}$}}
\put(11.38,2){\makebox(0,0){$a_{k}$}}
\put(7.38,2){\makebox(0,0){$v_{2}$}}
\put(7.38,0){\makebox(0,0){$v_{1}$}}
\put(7.65,1){\makebox(0,0){$v_{4}$}}
\put(8.8,1.3){\makebox(0,0){$v_{3}$}}
\put(6.5,1){\makebox(0,0){$T_{2k+4}$}}
\end{picture}\caption{Odd and even tree counterexamples.}%
\label{fig71}%
\end{figure}
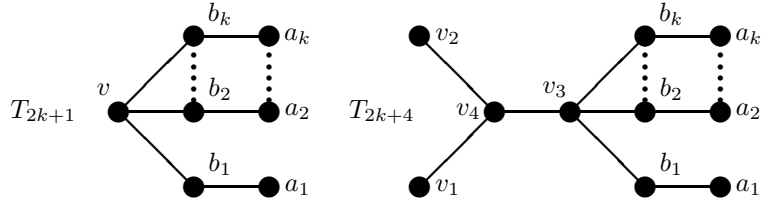

The tree $T_{2k+4},k\geq3$, has $n\left(  T_{2k+4}\right)  =2k+4$, $m\left(
T_{2k+4}\right)  =2k+3$, $\alpha\left(  T_{2k+1}\right)  =k+3$ and the degree
sequence
\[
\underset{k+2}{\underbrace{1,\cdots,1}},\underset{k}{\underbrace{2,\cdots,2}%
},3,k+1.
\]

Consequently, we infer that $h\left(  T_{2k+1}\right)  =k+2+\left\lfloor
\frac{k+1}{2}\right\rfloor >\max\left\{  \frac{n\left(  T_{2k+4}\right)  }%
{2},\alpha\left(  T_{2k+4}\right)  \right\}  $. According to Theorem
\ref{th2}\emph{(i)}, $T_{2k+4}$ is a strong unique maximum independence
tree\textit{. }Hence, $\Omega\left(  T_{2k+4}\right)  =\left\{  S\right\}  $,
where $S=\left\{  a_{1},...,a_{k},v_{1},v_{2},v_{3}\right\}  $. Since $S$ is
an independent vertex cover, by Lemma \ref{lem7}, we conclude that each
maximum independent set of $T_{2k+4}$ is a maximal non-maximum annihilating set.
\end{proof}

\section{Bipartite counterexamples that are not trees}

\begin{lemma}
\label{lem5}The bipartite graph $G_{k},k\geq1$, from Figure \ref{fig5555} has
$\alpha\left(  G_{k}\right)  =k+4$ and each $S\in\Omega\left(  G_{k}\right)  $
includes either $A_{0}=\left\{  a_{i}:i=1,...,4\right\}  $ or $B_{0}=\left\{
b_{i}:i=1,...,4\right\}  $.\textit{ Moreover, }every $S\in\Omega\left(
G_{k}\right)  $ but one contains $x_{k}$, and $\deg\left(  S\right)
\in\left\{  2k+11,2k+12\right\}  $.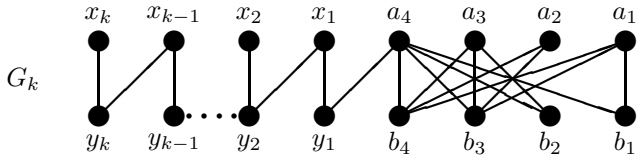
\begin{figure}[h]
\setlength{\unitlength}{1cm}\begin{picture}(5,2.2)\thicklines
\multiput(4,0.5)(1,0){8}{\circle*{0.29}}
\multiput(4,1.5)(1,0){8}{\circle*{0.29}}
\put(4,0.5){\line(0,1){1}}
\put(4,0.5){\line(1,1){1}}
\put(5,0.5){\line(0,1){1}}
\multiput(5,0.5)(0.2,0){5}{\circle*{0.07}}
\put(6,0.5){\line(0,1){1}}
\put(6,0.5){\line(1,1){1}}
\put(7,0.5){\line(0,1){1}}
\put(7,0.5){\line(1,1){1}}
\put(8,0.5){\line(0,1){1}}
\put(8,0.5){\line(1,1){1}}
\put(8,0.5){\line(2,1){2}}
\put(8,0.5){\line(3,1){3}}
\put(8,1.5){\line(1,-1){1}}
\put(8,1.5){\line(2,-1){2}}
\put(8,1.5){\line(3,-1){3}}
\put(9,0.5){\line(0,1){1}}
\put(9,0.5){\line(1,1){1}}
\put(9,0.5){\line(2,1){2}}
\put(9,1.5){\line(1,-1){1}}
\put(11,0.5){\line(0,1){1}}
\put(4,1.85){\makebox(0,0){$x_{k}$}}
\put(5,1.85){\makebox(0,0){$x_{k-1}$}}
\put(6,1.85){\makebox(0,0){$x_{2}$}}
\put(7,1.85){\makebox(0,0){$x_{1}$}}
\put(4,0.15){\makebox(0,0){$y_{k}$}}
\put(5,0.15){\makebox(0,0){$y_{k-1}$}}
\put(6,0.15){\makebox(0,0){$y_{2}$}}
\put(7,0.15){\makebox(0,0){$y_{1}$}}
\put(8,1.85){\makebox(0,0){$a_{4}$}}
\put(9,1.85){\makebox(0,0){$a_{3}$}}
\put(10,1.85){\makebox(0,0){$a_{2}$}}
\put(11,1.85){\makebox(0,0){$a_{1}$}}
\put(8,0.15){\makebox(0,0){$b_{4}$}}
\put(9,0.15){\makebox(0,0){$b_{3}$}}
\put(10,0.15){\makebox(0,0){$b_{2}$}}
\put(11,0.15){\makebox(0,0){$b_{1}$}}
\put(3,1){\makebox(0,0){$G_{k}$}}
\end{picture}\caption{A bipartite graph with $\alpha\left(  G_{k}\right)
=k+4,k\geq0$.}%
\label{fig5555}%
\end{figure}
\end{lemma}

\begin{proof}
Let $G_{0}=G_{k}-\left\{  x_{1},...x_{k},y_{1},...,y_{k}\right\}  =\left(
A_{0},B_{0},E_{0}\right)  $. Clearly, $\alpha\left(  G_{0}\right)  =4$ and
$A_{0},B_{0}\in\Omega\left(  G_{0}\right)  $, because $G_{0}$ has a perfect
matching. The spanning graph
\[
L=\left(  V\left(  G_{0}\right)  ,\left\{  a_{1}b_{4},b_{4}a_{2},a_{2}%
b_{3},b_{3}a_{3},a_{3}b_{2},b_{2}a_{4},a_{4}b_{1},b_{1}a_{1}\right\}  \right)
=\left(  A_{0},B_{0},U\right)
\]
of $G_{0}$ has exactly two maximum independent sets, namely $A_{0}$ and
$B_{0}$, as $L$ is isomorphic to $C_{8}$. The graph $G_{0}$ can be obtained
from $L$ by adding some edges keeping $A_{0},B_{0}$ as independent sets.
Consequently, we infer that $\Omega\left(  G_{0}\right)  =\Omega\left(
L\right)  =\left\{  A_{0},B_{0}\right\}  $.

Since $G_{k}$ is bipartite and has a perfect matching, we get that
$\alpha\left(  G_{k}\right)  =k+4$ for every $k\geq1$, and $\left\vert
S\cap\left(  \left\{  x_{i}:i=1,...,k\right\}  \cup\left\{  y_{i}%
:i=1,...,k\right\}  \right)  \right\vert =k$, for each $S\in\Omega\left(
G_{k}\right)  $.

Assume that there is some $S^{\prime}\in\Omega\left(  G_{k}\right)  $ such
that both $A_{1}=S^{\prime}\cap A_{0}\neq\emptyset$ and $B_{1}=S^{\prime}\cap
B_{0}\neq\emptyset$. Since
\[
S^{\prime}=\left(  S^{\prime}\cap\left(  \left\{  x_{i}:i=1,...,k\right\}
\cup\left\{  y_{i}:i=1,...,k\right\}  \right)  \right)  \cup A_{1}\cup B_{1},
\]
we get that $A_{1}\cup B_{1}$ is an independent set in $G_{0}$ of size $4$,
different from both $A_{0}$ and $B_{0}$, thus contradicting the fact that
$\Omega\left(  G_{0}\right)  =\left\{  A_{0},B_{0}\right\}  $. In conclusion,
every maximum independent set of $G_{k}$ includes either $A_{0}=\left\{
a_{i}:i=1,...,4\right\}  $ or $B_{0}=\left\{  b_{i}:i=1,...,4\right\}  $.

Let $S\in\Omega\left(  G_{k}\right)  $.

\textit{Case 1.} $y_{k}\in S$. Then, necessarily, we get that $\left\{
y_{i}:i=1,...,k-1\right\}  \cup B_{0}\subset S$, because $y_{1}a_{4}\in
E\left(  G_{k}\right)  $. Consequently, $\deg\left(  S\right)  =m\left(
G_{k}\right)  =2k+12$.

\textit{Case 2.} $y_{k}\notin S$. Then, necessarily, $x_{k}\in S$.

If $\left\{  x_{i}:i=1,...,k-1\right\}  \subset S$, we have two options:

\begin{itemize}
\item $S=\left\{  x_{i}:i=1,...,k\right\}  \cup A_{0}$, and therefore,
$\deg\left(  S\right)  =m\left(  G_{k}\right)  =2k+12$.

\item $S=\left\{  x_{i}:i=1,...,k\right\}  \cup B_{0}$, and therefore,
$\deg\left(  S\right)  =m\left(  G_{k}\right)  -1=2k+11$.
\end{itemize}

Otherwise, if $y_{j}\in S$, then $\left\{  y_{i}:i=1,...,j-1\right\}  \cup
B_{0}\subset S$, and hence, we get that $\deg\left(  S\right)  =2k+11$.
\end{proof}

\begin{theorem}
\label{th5}For every $k\geq0$, there exists a connected bipartite graph
$G_{k}$, of order $2k+8$, satisfying the following:

\begin{itemize}
\item $h\left(  G_{k}\right)  >\frac{n\left(  G_{k}\right)  }{2}=\alpha\left(
G_{k}\right)  $,

\item each $S\in\Omega\left(  G_{k}\right)  $ is a maximal annihilating set.
\end{itemize}
\end{theorem}

\begin{proof}
Let $G_{0}=G_{k}-\left\{  x_{1},...x_{k},y_{1},...,y_{k}\right\}  $, where
$G_{k}=(A_{k},B_{k},E_{k}),k\geq1$, is the graph from Figure \ref{fig5555},
while
\begin{align*}
A_{0}  &  =\left\{  a_{i}:i=1,2,3,4\right\}  ,B_{0}=\left\{  b_{i}%
:i=1,2,3,4\right\}  ,\\
A_{k}  &  =A_{0}\cup\left\{  x_{1},...,x_{k}\right\}  \text{ and }B_{k}%
=B_{0}\cup\left\{  y_{1},...,y_{k}\right\}  .
\end{align*}
Clearly, $\frac{n\left(  G_{k}\right)  }{2}=\alpha\left(  G_{k}\right)  $,
since every $G_{k}$ is a bipartite graph with a perfect matching.

\textit{Case 1}. $k=0$. The bipartite graph $G_{0}=\left(  A_{0}%
,B_{0},E\right)  $ has $m\left(  G_{0}\right)  =12$ and the degree sequence
$\left(  2,2,2,3,3,4,4,4\right)  $. Hence $h\left(  G_{0}\right)
=5>4=\alpha\left(  G_{0}\right)  $. In addition, $\Omega\left(  G_{0}\right)
=\left\{  A_{0},B_{0}\right\}  $ and $\deg\left(  A_{0}\right)  =\deg\left(
B_{0}\right)  =m\left(  G_{0}\right)  $, i.e., each $S\in\Omega\left(
G_{0}\right)  $ is a maximal non-maximum annihilating set.

\textit{Case 2}. $k=1$. The bipartite graph $G_{1}=\left(  A_{1}%
,B_{1},E\right)  $ has $m\left(  G_{1}\right)  =14$ and the degree sequence
$\left(  1,2,2,2,2,3,3,4,4,5\right)  $. Hence $h\left(  G_{1}\right)
=6>5=\alpha\left(  G_{1}\right)  $. According to Lemma \ref{lem5}, for every
$S\in$ $\Omega\left(  G_{1}\right)  $ we have that $\deg\left(  S\right)
\in\left\{  13,14\right\}  $, while
\[
\deg\left(  S\right)  +\min\left\{  \deg\left(  v\right)  :v\notin S\right\}
>14,
\]
i.e., each $S\in$ $\Omega\left(  G_{1}\right)  $ is a maximal non-maximum
annihilating set.

\textit{Case 3}. $k=2$. The bipartite graph $G_{2}=\left(  A_{2}%
,B_{2},E\right)  $ has $m\left(  G_{2}\right)  =16$ and the degree sequence
$\left(  1,2,2,2,2,2,2,3,3,4,4,5\right)  $. Hence $h\left(  G_{2}\right)
=8>6=\alpha\left(  G_{2}\right)  $. By Lemma \ref{lem5}, for every $S\in$
$\Omega\left(  G_{2}\right)  $ we have that $\deg\left(  S\right)  \in\left\{
15,16\right\}  $, while
\[
\deg\left(  S\right)  +\min\left\{  \deg\left(  v\right)  :v\notin S\right\}
>16,
\]
i.e., each $S\in$ $\Omega\left(  G_{2}\right)  $ is a maximal non-maximum
annihilating set.

\textit{Case 4}. $k\geq3$. The bipartite graph $G_{k}=\left(  A_{k}%
,B_{k},E\right)  $ has the degree sequence
\[
1,\underset{2k+2}{\underbrace{2,\cdots,2}},3,3,4,4,5.
\]
Thus $m\left(  G_{k}\right)  =2k+12$. It follows that the sum $1+2\left(
2k+2\right)  +3=4k+8$ of the first $2k+4$ degrees is greater than $m\left(
G_{k}\right)  $ for each $k\geq3$. Hence, $h\left(  G_{k}\right)  \leq2k+3$.

The inequality $1+2x\leq m\left(  G_{k}\right)  $ leads to $x\leq\frac
{2k+11}{2}$, which gives
\[
h\left(  G_{k}\right)  =1+\left\lfloor \frac{2k+11}{2}\right\rfloor
=k+6>k+4=\alpha\left(  G_{k}\right)  .
\]

Lemma \ref{lem5} claims that for every $S\in$ $\Omega\left(  G_{k}\right)  $
we have that $\deg\left(  S\right)  \in\left\{  2k+11,2k+12\right\}  $.

If $S\in$ $\Omega\left(  G_{k}\right)  $, then $\min\left\{  \deg\left(
v\right)  :v\notin S\right\}  =1$ if and only if $S=B_{0}\cup\left\{
y_{1},...,y_{k}\right\}  $. To this end, $S$\ is a vertex cover of $G_{k}$.
Thus, by Lemma \ref{lem7}, $S$\ is a maximal non-maximum annihilating set.

Otherwise, we get that
\[
\deg\left(  S\right)  +\min\left\{  \deg\left(  v\right)  :v\notin S\right\}
\geq\deg\left(  S\right)  +2>m\left(  G_{k}\right)  .
\]
Therefore, each $S\in$ $\Omega\left(  G_{k}\right)  $ is a maximal non-maximum
annihilating set.
\end{proof}

\begin{lemma}
\label{lem3}The graph $G_{k},k\geq0$, from Figure \ref{fig22} has a
\textit{unique maximum independent set, namely, }$A_{k}=\left\{
x_{k},...,x_{1},a_{5},a_{4},a_{3},a_{2},a_{1}\right\}  $, where $G_{0}%
=G_{k}-\left\{  x_{1},...x_{k},y_{1},...,y_{k}\right\}  $ and $A_{0}=\left\{
a_{5},a_{4},a_{3},a_{2},a_{1}\right\}  $\textit{.}\begin{figure}[h]
\setlength{\unitlength}{1cm}\begin{picture}(5,2.2)\thicklines
\multiput(3,0.5)(1,0){8}{\circle*{0.29}}
\multiput(3,1.5)(1,0){9}{\circle*{0.29}}
\put(3,0.5){\line(0,1){1}}
\put(3,0.5){\line(1,1){1}}
\put(4,0.5){\line(0,1){1}}
\multiput(4,0.5)(0.2,0){5}{\circle*{0.07}}
\put(5,0.5){\line(0,1){1}}
\put(5,0.5){\line(1,1){1}}
\put(6,0.5){\line(0,1){1}}
\put(6,0.5){\line(1,1){1}}
\put(7,0.5){\line(0,1){1}}
\put(7,0.5){\line(4,1){4}}
\put(8,0.5){\line(3,1){3}}
\put(7,0.5){\line(1,1){1}}
\put(7,0.5){\line(2,1){2}}
\put(7,0.5){\line(3,1){3}}
\put(7,1.5){\line(1,-1){1}}
\put(7,1.5){\line(2,-1){2}}
\put(7,1.5){\line(3,-1){3}}
\put(8,0.5){\line(0,1){1}}
\put(8,0.5){\line(1,1){1}}
\put(8,0.5){\line(2,1){2}}
\put(8,1.5){\line(1,-1){1}}
\put(10,0.5){\line(0,1){1}}
\put(3,1.85){\makebox(0,0){$x_{k}$}}
\put(4,1.85){\makebox(0,0){$x_{k-1}$}}
\put(5,1.85){\makebox(0,0){$x_{1}$}}
\put(6,1.85){\makebox(0,0){$x_{1}$}}
\put(3,0.15){\makebox(0,0){$y_{k}$}}
\put(4,0.15){\makebox(0,0){$y_{k-1}$}}
\put(5,0.15){\makebox(0,0){$y_{2}$}}
\put(6,0.15){\makebox(0,0){$y_{1}$}}
\put(7,1.85){\makebox(0,0){$a_{5}$}}
\put(8,1.85){\makebox(0,0){$a_{4}$}}
\put(9,1.85){\makebox(0,0){$a_{3}$}}
\put(10,1.85){\makebox(0,0){$a_{2}$}}
\put(11,1.85){\makebox(0,0){$a_{1}$}}
\put(7,0.15){\makebox(0,0){$b_{4}$}}
\put(8,0.15){\makebox(0,0){$b_{3}$}}
\put(9,0.15){\makebox(0,0){$b_{2}$}}
\put(10,0.15){\makebox(0,0){$b_{1}$}}
\put(2,1){\makebox(0,0){$G_{k}$}}
\end{picture}\caption{A bipartite graph of odd order with $\alpha\left(
G_{k}\right)  =k+5,k\geq0$.}%
\label{fig22}%
\end{figure}
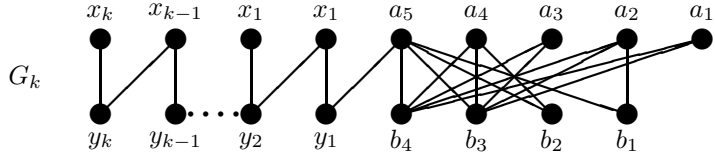
\end{lemma}

\begin{proof}
For every $k\geq0$, the set $\left\{  x_{j}y_{j}:j=1,...,k\right\}
\cup\left\{  a_{5}b_{4},a_{4}b_{2},a_{3}b_{3},a_{2}b_{1}\right\}  $ is a
maximum matching of $G_{k}$, which gives $\mu\left(  G_{k}\right)  =k+4$.
Since $\left\vert V\left(  G_{k}\right)  \right\vert =2k+9=\alpha\left(
G_{k}\right)  +\mu\left(  G_{k}\right)  $ and $A_{k}$ is an independent set of
cardinality $k+5$, we infer that $A_{k}\in\Omega\left(  G_{k}\right)  $ and
$\alpha\left(  G_{k}\right)  =k+5$.

Consider the trees $T_{k}=\left(  A_{k},V\left(  G_{k}\right)  -A_{k}%
,U\right)  ,k\geq1$ and $T_{0}=T_{k}-\{x_{j},y_{j}:1\leq j\leq k\}$ from
Figure \ref{fig122}. Since all the leaves of $T_{k}$ belong to $A_{k}$, it
follows that the distance between any two leaves is even. Consequently, by
Theorem \ref{th2}\textit{(i)}, we get that $T_{k}$ is a strong unique maximum
independence tree, and $\Omega\left(  T_{k}\right)  =\left\{  A_{k}\right\}
$. On the other hand, $T_{k}$ is a spanning tree of $G_{k}$, for every
$k\geq0$. Further, Theorem \ref{th2}\textit{(ii)} implies that $G_{k}$ is a
strong unique maximum independence graph and $\Omega\left(  G_{k}\right)
=\left\{  A_{k}\right\}  $.
\end{proof}

\begin{theorem}
\label{th6}For each $k\geq0$, there exists a connected bipartite graph $G_{k}$
of order $2k+9$, satisfying the following:

\begin{itemize}
\item $h\left(  G_{k}\right)  >\left\lceil \frac{n\left(  G_{k}\right)  }%
{2}\right\rceil =\alpha\left(  G_{k}\right)  $,

\item every $S\in\Omega\left(  G_{k}\right)  $ is a maximal annihilating set.
\end{itemize}
\end{theorem}

\begin{proof}
Let $G_{k}=(A_{k},B_{k},E_{k}),k\geq0$, be the bipartite graph from Figure
\ref{fig22}, where
\begin{align*}
A_{0}  &  =\left\{  a_{i}:i=1,2,3,4,5\right\}  ,B_{0}=\left\{  b_{i}%
:i=1,2,3,4\right\}  ,\\
A_{k}  &  =A_{0}\cup\left\{  x_{1},...,x_{k}\right\}  \text{ and }B_{k}%
=B_{0}\cup\left\{  y_{1},...,y_{k}\right\}  .
\end{align*}
Lemma \ref{lem3} claims that $A_{k}=\left\{  x_{1},...x_{k},a_{1}%
,...,a_{5}\right\}  $ is the unique maximum independent set of $G_{k}$. Hence,
$\left\lceil \frac{n\left(  G_{k}\right)  }{2}\right\rceil =\alpha\left(
G_{k}\right)  $, since $G_{k}-a_{1}$ is bipartite and has a perfect matching.

\textit{Case 1}. $k=0$. The bipartite graph $G_{0}=\left(  A_{0},B_{0}%
,E_{0}\right)  $ has $m\left(  G_{0}\right)  =14$ and the degree sequence
$\left(  2,2,2,2,3,3,4,5,5\right)  $. Hence, $h\left(  G_{0}\right)
=6>5=\alpha\left(  G_{0}\right)  $. In addition, $\deg\left(  A_{0}\right)
=m\left(  G_{0}\right)  $, i.e., each maximum independent set of $G_{0}$ is a
maximal non-maximum annihilating set.

\textit{Case 2}. $k=1$. The bipartite graph $G_{1}=\left(  A_{1},B_{1}%
,E_{1}\right)  $ has $m\left(  G_{1}\right)  =16$ and the degree sequence
$\left(  1,2,2,2,2,2,3,3,4,4,5,5\right)  $. Hence, $h\left(  G_{1}\right)
=7>6=\alpha\left(  G_{1}\right)  $. In addition, $\deg\left(  A_{1}\right)
=m\left(  G_{1}\right)  $, while
\[
\deg\left(  A_{1}\right)  +\min\left\{  \deg\left(  v\right)  :v\notin
A_{1}\right\}  \geq16+2>m\left(  G_{1}\right)  ,
\]
i.e., each maximum independent set of $G_{1}$ is a maximal non-maximum
annihilating set.

\textit{Case 3}. $k=2$. The bipartite graph $G_{2}=\left(  A_{2},B_{2}%
,E_{2}\right)  $ has $m\left(  G_{2}\right)  =18$ and the degree sequence
$\left(  1,2,2,2,2,2,2,2,3,3,5,5,5\right)  $. Hence, $h\left(  G_{2}\right)
=9>7=\alpha\left(  G_{2}\right)  $. In addition, $\deg\left(  A_{2}\right)
=m\left(  G_{2}\right)  $, while
\[
\deg\left(  A_{2}\right)  +\min\left\{  \deg\left(  v\right)  :v\notin
A_{2}\right\}  \geq18+2>m\left(  G_{2}\right)  ,
\]
i.e., each maximum independent set of $G_{2}$ is a maximal non-maximum
annihilating set.

\textit{Case 4}. $k\geq3$. The bipartite graph $G_{k}=\left(  A_{k}%
,B_{k},E_{k}\right)  $ has the degree sequence
\[
1,\underset{2k+3}{\underbrace{2,\cdots,2}},3,3,5,5,5.
\]
Thus $m\left(  G_{k}\right)  =2k+14$. It follows that the sum $1+2\left(
2k+3\right)  +3=4k+10$ of the first $2k+5$ degrees is greater than $m\left(
G_{k}\right)  $ for each $k\geq3$. Hence, $h\left(  G_{k}\right)  \leq2k+4$.

The inequality $1+2x\leq m\left(  G_{k}\right)  $ leads to $x\leq\frac
{2k+13}{2}$, which gives
\[
h\left(  G_{k}\right)  =1+\left\lfloor \frac{2k+13}{2}\right\rfloor
=k+7>k+5=\alpha\left(  G_{k}\right)  .
\]

In addition, $\deg\left(  A_{k}\right)  =m\left(  G_{k}\right)  $, while
\[
\deg\left(  A_{k}\right)  +\min\left\{  \deg\left(  v\right)  :v\notin
A_{k}\right\}  =2k+16>m\left(  G_{k}\right)  ,
\]
i.e., each maximum independent set of $G_{k}$ is maximal non-maximum annihilating.
\end{proof}

\section{Non-bipartite K\"{o}nig-Egerv\'{a}ry counterexamples}

In what follows, we present a series of counterexamples to the opposite
direction of Conjecture \ref{Conj1} for non-bipartite K\"{o}nig-Egerv\'{a}ry
graphs. All these graphs have unique maximum independent sets.

\begin{lemma}
\label{lem6}The graph $H_{k},k\geq0$, from Figure \ref{fig717} is a
K\"{o}nig-Egerv\'{a}ry graph that has a unique maximum independent
set\textit{, namely, }$S_{k}=\left\{  x_{k},...,x_{1},a_{4},a_{3},a_{2}%
,a_{1}\right\}  $, where $H_{0}=H_{k}-\left\{  x_{j},y_{j}%
:j=1,2,...,k\right\}  $\ and $S_{0}=\left\{  a_{4},a_{3},a_{2},a_{1}\right\}
$.\begin{figure}[h]
\setlength{\unitlength}{0.85cm}\begin{picture}(5,4)\thicklines
\multiput(2,1)(2,0){7}{\circle*{0.29}}
\multiput(2,3)(2,0){7}{\circle*{0.29}}
\multiput(4,1)(0.2,0){10}{\circle*{0.12}}
\multiput(4,3)(0.2,0){10}{\circle*{0.12}}
\put(6,1){\line(1,0){4}}
\put(2,1){\line(0,1){2}}
\put(2,1){\line(1,0){2}}
\put(2,1){\line(1,1){2}}
\put(2,1){\line(2,1){4}}
\put(2,1){\line(3,1){6}}
\put(2,1){\line(4,1){8}}
\put(2,1){\line(5,1){10}}
\put(2,1){\line(6,1){12}}
\put(2,3){\line(1,-1){2}}
\put(2,3){\line(2,-1){4}}
\put(2,3){\line(3,-1){6}}
\put(2,3){\line(4,-1){8}}
\put(2,3){\line(5,-1){10}}
\put(2,3){\line(6,-1){12}}
\put(4,1){\line(0,1){2}}
\put(4,1){\line(1,1){2}}
\put(4,1){\line(2,1){4}}
\put(4,1){\line(3,1){6}}
\put(4,1){\line(4,1){8}}
\put(4,1){\line(5,1){10}}
\put(4,3){\line(1,-1){2}}
\put(4,3){\line(2,-1){4}}
\put(4,3){\line(3,-1){6}}
\put(4,3){\line(4,-1){8}}
\put(4,3){\line(5,-1){10}}
\put(6,1){\line(0,1){2}}
\put(6,1){\line(1,1){2}}
\put(6,1){\line(2,1){4}}
\put(6,1){\line(3,1){6}}
\put(6,1){\line(4,1){8}}
\put(6,3){\line(1,-1){2}}
\put(6,3){\line(2,-1){4}}
\put(6,3){\line(3,-1){6}}
\put(6,3){\line(4,-1){8}}
\put(14,1){\line(0,1){2}}
\put(8,1){\line(0,1){2}}
\put(8,1){\line(1,1){2}}
\put(8,1){\line(2,1){4}}
\put(8,1){\line(3,1){6}}
\put(8,3){\line(1,-1){2}}
\put(8,3){\line(2,-1){4}}
\put(8,3){\line(3,-1){6}}
\put(10,3){\line(1,-1){2}}
\put(10,1){\line(0,1){2}}
\put(10,1){\line(1,1){2}}
\put(10,1){\line(2,1){4}}
\put(14,3.35){\makebox(0,0){$a_{1}$}}
\put(12,3.35){\makebox(0,0){$a_{2}$}}
\put(10,3.35){\makebox(0,0){$a_{3}$}}
\put(8,3.35){\makebox(0,0){$a_{4}$}}
\put(14,3.75){\makebox(0,0){$k+3$}}
\put(12,3.75){\makebox(0,0){$k+2$}}
\put(10,3.75){\makebox(0,0){$k+3$}}
\put(8,3.75){\makebox(0,0){$k+4$}}
\put(6,3.35){\makebox(0,0){$x_{1}$}}
\put(4,3.35){\makebox(0,0){$x_{k-1}$}}
\put(2,3.35){\makebox(0,0){$x_{k}$}}
\put(6,3.75){\makebox(0,0){$k+4$}}
\put(4,3.75){\makebox(0,0){$k+4$}}
\put(2,3.75){\makebox(0,0){$k+4$}}
\put(14,0.45){\makebox(0,0){$b_{1}$}}
\put(12,0.45){\makebox(0,0){$b_{2}$}}
\put(10,0.45){\makebox(0,0){$b_{3}$}}
\put(8,0.45){\makebox(0,0){$b_{4}$}}
\put(8,0){\makebox(0,0){$k+6$}}
\put(10,0){\makebox(0,0){$k+5$}}
\put(12,0){\makebox(0,0){$k+2$}}
\put(14,0){\makebox(0,0){$k+2$}}
\put(6,0.45){\makebox(0,0){$y_{1}$}}
\put(4,0.45){\makebox(0,0){$y_{k-1}$}}
\put(6,0){\makebox(0,0){$k+6$}}
\put(4,0){\makebox(0,0){$k+6$}}
\put(2,0.45){\makebox(0,0){$y_{k}$}}
\put(2,0){\makebox(0,0){$k+5$}}
\put(0.7,2){\makebox(0,0){$G_{k}$}}
\end{picture}\caption{$H_{k}$ is a non-bipartite K\"{o}nig-Egerv\'{a}ry graph
with $\alpha\left(  H_{k}\right)  =k+4,k\geq0$.}%
\label{fig717}%
\end{figure}
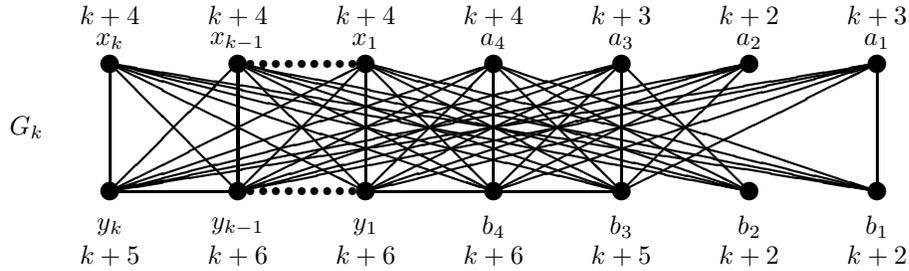
\end{lemma}

\begin{proof}
Clearly, $S_{k}=\left\{  x_{k},...,x_{1},a_{4},a_{3},a_{2},a_{1}\right\}  $ is
an independent set and
\[
\left\{  x_{j}y_{j}:j=1,2,...,k\right\}  \cup\left\{  a_{4}b_{4},a_{3}%
b_{2},a_{3}b_{3},a_{1}b_{1}\right\}
\]
is a perfect matching of $H_{k}$. Hence, we get%
\[
\left\vert V_{k}\right\vert =2\mu\left(  H_{k}\right)  =\left\vert
S_{k}\right\vert +\mu\left(  H_{k}\right)  \leq\alpha\left(  H_{k}\right)
+\mu\left(  H_{k}\right)  \leq\left\vert V_{k}\right\vert ,
\]
which implies $\alpha\left(  H_{k}\right)  +\mu\left(  H_{k}\right)
=\left\vert V_{k}\right\vert $, i.e., $H_{k}$ is a K\"{o}nig-Egerv\'{a}ry
graph, and $\alpha\left(  H_{k}\right)  =k+4=\left\vert S_{k}\right\vert $.

Let $L_{k}=H_{k}\left[  X_{k}\cup Y_{k}\right]  ,k\geq1$, and $L_{0}%
=H_{k}\left[  A\cup B\right]  $, where
\begin{align*}
X_{k}  &  =\left\{  x_{j}:j=1,...,k\right\}  ,Y_{k}=\left\{  y_{j}%
:j=1,...,k\right\}  ,\\
A  &  =\left\{  a_{1},a_{2},a_{3},a_{4}\right\}  \text{ and }B=\left\{
b_{1},b_{2},b_{3},b_{4}\right\}  .
\end{align*}
Since $L_{k}$ has, on the one hand, $K_{k,k}$ as a subgraph, and, on the other
hand,
\[
y_{k}y_{k-1},y_{k-1}y_{k-2},...,y_{2}y_{1}\in E\left(  L_{k}\right)  ,
\]
it follows that $X_{k}$ is the unique maximum independent set of $L_{k}$.

The graph $L_{0}$ has $A$ as a unique independent set, because
\[
C_{8}+b_{3}b_{4}=\left(  A\cup B,\left\{  a_{1}b_{4},b_{4}a_{2},a_{2}%
b_{3},b_{3}a_{3},a_{3}b_{2},b_{2}a_{4},a_{4}b_{1},b_{1}a_{1},b_{3}%
b_{4}\right\}  \right)
\]
has $A$ as a unique maximum independent set, and $L_{0}$ can be obtained from
$C_{8}+b_{3}b_{4}$ by adding a number of edges.

Since $H_{k}$ can be obtained from the union of $L_{k}$ and $L_{0}$ by adding
some edges, and $S_{k}=X_{k}\cup A$ is independent in $H_{k}$, it follows that
$H_{k}$ has $S_{k}$ as a unique maximum independent set.
\end{proof}

\begin{corollary}
\label{Cor1}The graph $H_{k},k\geq0$, from Figure \ref{fig727} is a
K\"{o}nig-Egerv\'{a}ry graph that has a unique \textit{independent set,
namely, }$S_{k}=\left\{  x_{i}:i=1,...,k\right\}  \cup\left\{  a_{i}%
:i=1,...,5\right\}  $\textit{, where }$H_{0}=H_{k}-\left\{  x_{j}%
,y_{j}:j=1,2,...,k\right\}  $\ and\textit{ }$S_{0}=\left\{  a_{i}%
:i=1,...,5\right\}  $\textit{.}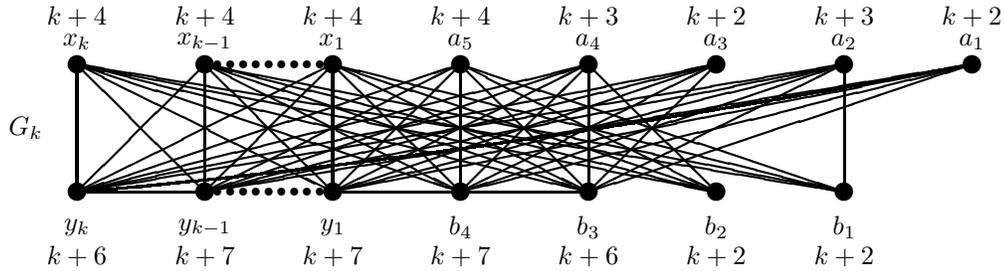
\begin{figure}[h]
\setlength{\unitlength}{0.85cm}\begin{picture}(5,4)\thicklines
\multiput(1,1)(2,0){7}{\circle*{0.29}}
\multiput(1,3)(2,0){8}{\circle*{0.29}}
% \put(1,1){\line(15,3){14}}
\qbezier(1,1)(7,1.5)(15,3)
\put(3,1){\line(6,1){12}}
\put(5,1){\line(5,1){10}}
\put(7,1){\line(4,1){8}}
\put(9,1){\line(3,1){6}}
\multiput(3,1)(0.2,0){10}{\circle*{0.12}}
\multiput(3,3)(0.2,0){10}{\circle*{0.12}}
\put(5,1){\line(1,0){4}}
\put(1,1){\line(0,1){2}}
\put(1,1){\line(1,0){2}}
\put(1,1){\line(1,1){2}}
\put(1,1){\line(2,1){4}}
\put(1,1){\line(3,1){6}}
\put(1,1){\line(4,1){8}}
\put(1,1){\line(5,1){10}}
\put(1,1){\line(6,1){12}}
\put(1,3){\line(1,-1){2}}
\put(1,3){\line(2,-1){4}}
\put(1,3){\line(3,-1){6}}
\put(1,3){\line(4,-1){8}}
\put(1,3){\line(5,-1){10}}
\put(1,3){\line(6,-1){12}}
\put(3,1){\line(0,1){2}}
\put(3,1){\line(1,1){2}}
\put(3,1){\line(2,1){4}}
\put(3,1){\line(3,1){6}}
\put(3,1){\line(4,1){8}}
\put(3,1){\line(5,1){10}}
\put(3,3){\line(1,-1){2}}
\put(3,3){\line(2,-1){4}}
\put(3,3){\line(3,-1){6}}
\put(3,3){\line(4,-1){8}}
\put(3,3){\line(5,-1){10}}
\put(5,1){\line(0,1){2}}
\put(5,1){\line(1,1){2}}
\put(5,1){\line(2,1){4}}
\put(5,1){\line(3,1){6}}
\put(5,1){\line(4,1){8}}
\put(5,3){\line(1,-1){2}}
\put(5,3){\line(2,-1){4}}
\put(5,3){\line(3,-1){6}}
\put(5,3){\line(4,-1){8}}
\put(13,1){\line(0,1){2}}
\put(7,1){\line(0,1){2}}
\put(7,1){\line(1,1){2}}
\put(7,1){\line(2,1){4}}
\put(7,1){\line(3,1){6}}
\put(7,3){\line(1,-1){2}}
\put(7,3){\line(2,-1){4}}
\put(7,3){\line(3,-1){6}}
\put(9,3){\line(1,-1){2}}
\put(9,1){\line(0,1){2}}
\put(9,1){\line(1,1){2}}
\put(9,1){\line(2,1){4}}
\put(15,3.35){\makebox(0,0){$a_{1}$}}
\put(13,3.35){\makebox(0,0){$a_{2}$}}
\put(11,3.35){\makebox(0,0){$a_{3}$}}
\put(9,3.35){\makebox(0,0){$a_{4}$}}
\put(7,3.35){\makebox(0,0){$a_{5}$}}
\put(15,3.75){\makebox(0,0){$k+2$}}
\put(13,3.75){\makebox(0,0){$k+3$}}
\put(11,3.75){\makebox(0,0){$k+2$}}
\put(9,3.75){\makebox(0,0){$k+3$}}
\put(7,3.75){\makebox(0,0){$k+4$}}
\put(5,3.35){\makebox(0,0){$x_{1}$}}
\put(3,3.35){\makebox(0,0){$x_{k-1}$}}
\put(1,3.35){\makebox(0,0){$x_{k}$}}
\put(5,3.75){\makebox(0,0){$k+4$}}
\put(3,3.75){\makebox(0,0){$k+4$}}
\put(1,3.75){\makebox(0,0){$k+4$}}
\put(13,0.45){\makebox(0,0){$b_{1}$}}
\put(11,0.45){\makebox(0,0){$b_{2}$}}
\put(9,0.45){\makebox(0,0){$b_{3}$}}
\put(7,0.45){\makebox(0,0){$b_{4}$}}
\put(7,0){\makebox(0,0){$k+7$}}
\put(9,0){\makebox(0,0){$k+6$}}
\put(11,0){\makebox(0,0){$k+2$}}
\put(13,0){\makebox(0,0){$k+2$}}
\put(5,0.45){\makebox(0,0){$y_{1}$}}
\put(3,0.45){\makebox(0,0){$y_{k-1}$}}
\put(5,0){\makebox(0,0){$k+7$}}
\put(3,0){\makebox(0,0){$k+7$}}
\put(1,0.45){\makebox(0,0){$y_{k}$}}
\put(1,0){\makebox(0,0){$k+6$}}
\put(0.2,2){\makebox(0,0){$G_{k}$}}
\end{picture}\caption{$H_{k}$ is a non-bipartite K\"{o}nig-Egerv\'{a}ry graph
with $\alpha\left(  G_{k}\right)  =k+5,k\geq0$.}%
\label{fig727}%
\end{figure}
\end{corollary}

\begin{proof}
According to Lemma \ref{lem6}, $H_{k}-a_{5}$ is a K\"{o}nig-Egerv\'{a}ry graph
with a unique maximum independent set, namely, $W_{k}=\left\{  x_{i}%
:i=1,...,k\right\}  \cup\left\{  a_{i}:i=1,...,4\right\}  $\textit{.}
Since\textit{ }$S_{k}=W_{k}\cup\left\{  a_{5}\right\}  $ is an independent set
and $\mu\left(  H_{k}\right)  =\mu\left(  H_{k}-a_{5}\right)  =k+4$, it
follows that $H_{k}$ is a K\"{o}nig-Egerv\'{a}ry graph and $S_{k}$ is its
unique maximum independent set.
\end{proof}

The following results show that if the order of the graph is greater or equal
to $8$, then the converse of Theorem \ref{Th9} is not true for non-bipartite
K\"{o}nig-Egerv\'{a}ry graphs.

\begin{theorem}
\label{Th10}For every $k\geq0$, there exists a connected non-bipartite
K\"{o}nig-Egerv\'{a}ry graph $H_{k}=(V_{k},E_{k})$, of order $2k+8$,
satisfying the following:

\begin{itemize}
\item $h\left(  H_{k}\right)  >\frac{n\left(  H_{k}\right)  }{2}=\alpha\left(
H_{k}\right)  $,

\item each $S\in\Omega\left(  H_{k}\right)  $ is a maximal annihilating set.
\end{itemize}
\end{theorem}

\begin{proof}
Let $H_{k}=(V_{k},E_{k}),k\geq0$, be the graph from Figure \ref{fig717} (in
the bottom and the top lines are written the degrees of its vertices), where
$H_{0}=H_{k}-\left\{  x_{1},...x_{k},y_{1},...,y_{k}\right\}  $. Clearly,
every $H_{k}$ is non-bipartite.

By Lemma \ref{lem6}, each $H_{k},k\geq0$, is a K\"{o}nig-Egerv\'{a}ry graph
with a unique maximum independent set, namely, $S_{k}=\left\{  x_{k}%
,...,x_{1},a_{4},a_{3},a_{2},a_{1}\right\}  $, where $S_{0}=\left\{
a_{4},a_{3},a_{2},a_{1}\right\}  $.

\textit{Case 1}. $k=0$. Since $m\left(  H_{0}\right)  =13$ and the degree
sequence $\left(  2,2,2,3,3,4,5,5\right)  $, we infer that $h\left(
H_{0}\right)  =5>4=\alpha\left(  H_{0}\right)  $. In addition, $\deg\left(
S_{0}\right)  =m\left(  H_{0}\right)  -1$, i.e., each maximum independent set
of $H_{0}$ is a maximal non-maximum annihilating set.

\textit{Case 2}. $k\geq1$.

Clearly, $H_{k}$ has $m\left(  G_{k}\right)  =k^{2}+9k+13$ and its degree
sequence is
\[
k+2,k+2,k+2,k+3,k+3,\underset{k+1}{\underbrace{k+4,\cdots,k+4}}%
,k+5,k+5,\underset{k}{\underbrace{k+6,\cdots,k+6}}.
\]
Since the sum of the first $k+6$ degrees of the sequence satisfies
\[
k^{2}+10k+16>m\left(  H_{k}\right)  ,
\]
we infer that the annihilation number $h\left(  H_{k}\right)  \leq k+6$. The
sum $12+4\left(  x-5\right)  +kx$ of the first $x\geq5$ degrees of the
sequence satisfies $12+4\left(  x-5\right)  +kx\leq m\left(  H_{k}\right)  $
for $x\leq\frac{k^{2}+9k+21}{k+4}$. This implies
\[
h\left(  H_{k}\right)  =\left\lfloor \frac{k^{2}+9k+21}{k+4}\right\rfloor
=k+5>k+4=\alpha\left(  H_{k}\right)  \text{,}%
\]
i.e., $H_{k}$ has no maximum annihilating set belonging to $\Omega\left(
H_{k}\right)  $. Since its unique maximum independent set $S_{k}=\left\{
a_{1},a_{2},a_{3},a_{4},x_{1},x_{2},...,x_{k}\right\}  $ has
\begin{gather*}
\deg\left(  S_{k}\right)  =k^{2}+8k+12<m\left(  H_{k}\right)  \text{, while}\\
\deg\left(  S_{k}\right)  +\min\{\deg\left(  v\right)  :v\in V_{k}-S\}=\left(
k^{2}+8k+12\right)  +\left(  k+2\right)  >m\left(  H_{k}\right)  ,
\end{gather*}
we infer that $S_{k}$ is a maximal annihilating set.
\end{proof}

\begin{theorem}
\label{Th11}For every $k\geq0$, there exists a connected non-bipartite
K\"{o}nig-Egerv\'{a}ry graph $H_{k}=(V_{k},E_{k})$, of order $2k+9$,
satisfying the following:

\begin{itemize}
\item $h\left(  H_{k}\right)  >\left\lceil \frac{n\left(  H_{k}\right)  }%
{2}\right\rceil =\alpha\left(  H_{k}\right)  $,

\item each $S\in\Omega\left(  H_{k}\right)  $ is a maximal annihilating set.
\end{itemize}
\end{theorem}

\begin{proof}
Let $H_{k}=(V_{k},E_{k}),k\geq1$, be the graph from Figure \ref{fig727} (in
the bottom and the top lines are written the degrees of its vertices), and
$H_{0}=H_{k}-\left\{  x_{1},...x_{k},y_{1},...,y_{k}\right\}  $.

Corollary \ref{Cor1} claims that $H_{k},k\geq0$, is a K\"{o}nig-Egerv\'{a}ry
graph with a unique maximum independent set, namely $S_{k}=\left\{
x_{1},...x_{k},a_{1},...,a_{5}\right\}  ,k\geq1$, and $S_{0}=\left\{
a_{1},...,a_{5}\right\}  $.

\textit{Case 1}. The non-bipartite K\"{o}nig-Egerv\'{a}ry graph $H_{0}$ has
$m\left(  H_{0}\right)  =15$ and the degree sequence $\left(
2,2,2,2,3,3,4,6,6\right)  $. Hence, $h\left(  H_{0}\right)  =6>5=\alpha\left(
G_{0}\right)  $. In addition, $\Omega\left(  H_{0}\right)  =\left\{
S_{0}\right\}  $, and $\deg\left(  S_{0}\right)  =14$, i.e., each maximum
independent set of $H_{0}$ is a maximal non-maximum annihilating set.

\textit{Case 2}. $k\geq1$.

Clearly, $H_{k}$ has $m\left(  H_{k}\right)  =k^{2}+10k+15$ and its degree
sequence is
\[
k+2,k+2,k+2,k+2,k+3,k+3,\underset{k+1}{\underbrace{k+4,\cdots,k+4}%
},k+5,k+5,\underset{k}{\underbrace{k+6,\cdots,k+6}}.
\]
Since the sum of the first $k+7$ degrees of the sequence satisfies \
\[
k^{2}+11k+18>m\left(  H_{k}\right)  ,
\]
we infer that the annihilation number $h\left(  H_{k}\right)  \leq k+6$. The
sum $14+4\left(  x-5\right)  +kx$ of the first $x\geq6$ degrees of the
sequence satisfies $14+4\left(  x-6\right)  +kx\leq m\left(  H_{k}\right)  $
for $x\leq\frac{k^{2}+10k+25}{k+4}$. This implies
\[
h\left(  H_{k}\right)  =\left\lfloor \frac{k^{2}+10k+25}{k+4}\right\rfloor
=k+6>k+5=\alpha\left(  H_{k}\right)  \text{,}%
\]
i.e., $H_{k}$ has no maximum annihilating set belonging to $\Omega\left(
H_{k}\right)  $. Since its unique maximum independent set $S_{k}$ has
\begin{gather*}
\deg\left(  S_{k}\right)  =k^{2}+9k+14<m\left(  G_{k}\right)  \text{, while}\\
\deg\left(  S_{k}\right)  +\min\{\deg\left(  v\right)  :v\in V_{k}%
-S_{k}\}=\left(  k^{2}+9k+14\right)  +\left(  k+2\right)  >m\left(
G_{k}\right)  ,
\end{gather*}
we infer that $S_{k}$ is a maximal annihilating set.
\end{proof}

\section{Conclusions}

If $G$ is a K\"{o}nig-Egerv\'{a}ry graph with $\alpha\left(  G\right)
\in\left\{  1,2\right\}  $, then $\alpha\left(  G\right)  =h\left(  G\right)
$ and each maximum independent set is maximal annihilating, since the list of
such K\"{o}nig-Egerv\'{a}ry graphs reads as follows:%

\[
\left\{  K_{1},K_{2},K_{1}\cup K_{1},K_{1}\cup K_{2},K_{2}\cup K_{2}%
,P_{3},P_{4},C_{4},K_{3}+e,K_{4}-e\right\}  .
\]
Consequently, Conjecture \ref{Conj1} is correct for K\"{o}nig-Egerv\'{a}ry
graphs with $\alpha\left(  G\right)  \leq2$.\begin{figure}[h]
\setlength{\unitlength}{1cm}\begin{picture}(5,1)\thicklines
\multiput(5,1)(1,0){2}{\circle*{0.29}}
\multiput(5,0)(1,0){2}{\circle*{0.29}}
\put(5,0){\line(1,1){1}}
\put(5,0){\line(0,1){1}}
\put(5,0){\line(1,0){1}}
\put(6,0){\line(0,1){1}}
\put(4.2,0.5){\makebox(0,0){$G_{1}$}}
\multiput(8,0)(1,0){2}{\circle*{0.29}}
\multiput(8,1)(1,0){2}{\circle*{0.29}}
\put(8,0){\line(1,1){1}}
\put(8,1){\line(1,-1){1}}
\put(8,0){\line(0,1){1}}
\put(8,0){\line(1,0){1}}
\put(9,0){\line(0,1){1}}
\put(7.2,0.5){\makebox(0,0){$G_{2}$}}
\end{picture}\caption{$G_{1}=K_{3}+e$ and $G_{2}=K_{4}-e$.}%
\label{fig88}%
\end{figure}
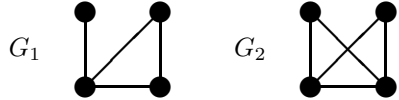

Let $G$ be a disconnected K\"{o}nig-Egerv\'{a}ry graph with $\alpha\left(
G\right)  =3$.

\begin{itemize}
\item If $\alpha\left(  G\right)  =h\left(  G\right)  $, then
\[
G\in\left\{
\begin{array}
[c]{c}%
3K_{1},2K_{1}\cup K_{2},K_{1}\cup2K_{2},3K_{2},K_{1}\cup P_{3},K_{1}\cup
P_{4},\\
K_{1}\cup C_{4},K_{1}\cup\left(  K_{3}+e\right)  ,K_{1}\cup\left(
K_{4}-e\right)  ,K_{2}\cup P_{3},K_{2}\cup C_{4}%
\end{array}
\right\}  ,
\]
while every $S\in\Omega\left(  G\right)  $ is a maximal annihilating set.

\item If $\alpha\left(  G\right)  <h\left(  G\right)  $, then $G\in\left\{
K_{2}\cup P_{4},K_{2}\cup\left(  K_{3}+e\right)  ,K_{2}\cup\left(
K_{4}-e\right)  \right\}  $, while for every such $G$, there exists a maximum
independent set, which is a not a maximal annihilating set. Moreover, for
$K_{2}\cup\left(  K_{3}+e\right)  $ and $K_{2}\cup\left(  K_{4}-e\right)  $
all maximum independent sets are not maximal annihilating.
\end{itemize}

Thus Conjecture \ref{Conj1} is true for disconnected K\"{o}nig-Egerv\'{a}ry
graphs with $\alpha\left(  G\right)  =3$.

On the other hand, Theorems \ref{Th12}, \ref{th5}, \ref{th6}, \ref{Th10},
\ref{Th11} present various counterexamples to the \textquotedblright%
\textit{(ii)} $\Longrightarrow$ \textit{(i)}\textquotedblleft\ part of
Conjecture \ref{Conj1} for every independence number greater than three.

\begin{conjecture}
Let $G$ be a graph with $h\left(  G\right)  \geq\frac{n\left(  G\right)  }{2}%
$. If $G$ is a connected K\"{o}nig-Egerv\'{a}ry graph with $\alpha\left(
G\right)  =3$, and every $S\in\Omega(G)$ is a maximal annihilating set, then
$\alpha\left(  G\right)  =h\left(  G\right)  $.
\end{conjecture}

\end{document}